\documentclass{article}
\usepackage{amssymb}
\usepackage{amsmath}
\newcommand{\dar}{{\downarrow}}
\newcommand{\pole}{\perp\!\!\!\!\perp}
\newcommand{\cc}{{\sf c}\! {\sf c}}

\newcommand{\cmb}[1]{{\bf\sf #1}}

\newtheorem{proposition}{Proposition}[section]
\newtheorem{lemma}[proposition]{Lemma}
\newtheorem{corollary}[proposition]{Corollary}
\newtheorem{definition}[proposition]{Definition}
\newtheorem{theorem}[proposition]{Theorem}
\newtheorem{example}[proposition]{Example}
\newtheorem{remark}[proposition]{Remark}
\newcounter{lijst-teller}
\newcounter{boon}
\newcounter{boon2}
\newenvironment{proof}{\noindent {\bf Proof}. \nopagebreak }{\nopagebreak\hfill\rule{2mm}{3mm}}

\input xypic
\xyoption{all}
\UseComputerModernTips
\CompileMatrices
\definemorphism{dat}\dashed\tip\notip
\definemorphism{dub}\Solid\Tip\notip

\newenvironment{rlist}%
   {\begin{list}{\roman{lijst-teller}\/)\hfil}%
              {\labelwidth 2em%
               \leftmargin\labelwidth\advance\leftmargin by\labelsep%
               \usecounter{lijst-teller}}}%
   {\end{list}}

\newenvironment{r'list}%
   {\begin{list}{\roman{lijst-teller}\/)$'$\hfil}%
              {\labelwidth 2em%
               \leftmargin\labelwidth\advance\leftmargin by\labelsep%
               \usecounter{lijst-teller}}}%
   {\end{list}}

\newenvironment{alist}
   {\begin{list}{\alph{boon}\/)\hfil}%
               {\labelwidth 2em%
               \leftmargin\labelwidth\advance\leftmargin by\labelsep%
               \usecounter{boon}}}%
   {\end{list}}

\newenvironment{arlist}
    {\begin{list}{\arabic{boon2}\/)\hfil}%
                 {\labelwidth 2em%
                 \leftmargin\labelwidth\advance\leftmargin by\labelsep%
                 \usecounter{boon2}}}%
    {\end{list}}

\title{Classical and Relative Realizability}
\author{Jaap van Oosten\footnote{Corresponding author. Department of Mathematics, Utrecht University. E-mail address: {\tt j.vanoosten@uu.nl}} \and Tingxiang Zou\footnote{presently at Universit\'e de Lyon 1, groupe de Logique}}
\date{March 8, 2016}

\begin{document}

\maketitle
\begin{abstract}
\noindent We show that every abstract Krivine structure in the sense of Streicher can be obtained, up to equivalence of the resulting tripos, from a filtered opca $(A,A')$ and a subobject of 1 in the relative realizability topos ${\sf RT}(A',A)$; the topos is always a Booleanization of a closed subtopos of ${\sf RT}(A',A)$. We exhibit a range of non-localic Boolean subtoposes of the Kleene-Vesley topos. 
\end{abstract}

\noindent {\bf Keywords}: realizability toposes, partial combinatory algebras, geometric morphisms, local operators, abstract Krivine structures, non-localic Boolean toposes.
\section*{Introduction}
In an impressive series of papers, Jean-Louis Krivine has been emplying extensions of the untyped $\lambda$-calculus to create ``realizability interpretations'' for classical ZF set theory. He has been working on this project for roughly the last 20 years.

For a long time, this work seemed to have no connections with other interpretations, also called `realizability', in the Kleene-Troelstra-Hyland tradition (for an overview of which, see e.g.\ \cite{OostenJ:reahe}). And disjoint research groups worked either in `Krivine realizability' or with notions of realizability related to the effective topos or similar toposes.

This situation has recently undergone a drastic change: the series of `realizability' meetings at Chamb\'ery has brought researchers from different traditions together, and in particular Thomas Streicher, who published  \cite{StreicherT:kcrcp}, has built an important bridge.

After reformulating Krivine's `abstract machine' as an `abstract Krivine structure' (aks), Streicher proves that from each aks one may construct a so-called {\sl filtered order-pca\/} (a structure for what is called ``relative realizability'' in Birkedal's thesis \cite{BirkedalL:devttc-entcs} and in \cite{BirkedalL:relmrr}), and hence a topos; the special features of the filtered opca constructed from an aks ensure that this will be a Boolean topos.

In a series of papers from 2013--2015 (\cite{FerrerSantosW:repr,FerrerSantosW:ordcar,FerrerSantosW:reaoa}) Walter Ferrer Santos, Jonas Frey, Mauricio Guillermo, Octavio Malherbe and Alexandre Miquel develop theory of ordered pcas whose associated Set-indexed preorders are Boolean triposes. Frey, moreover, investigated variations corresponding to different flavours of Krivine realizability (\cite{FreyJ:reats}).

All this work is, however, essentially syntactic. The focus of the present paper is on a {\em mathematical construction\/} of abstract Krivine structures.

We start with the concept of a {\sl Basic Combinatorial Object\/} from Pieter Hofstra's elegant paper \cite{HofstraP:allrr}. BCOs form a preorder-enriched category with a KZ-monad $\cal D$ (we rehearse the material we need in section \ref{BCOsection}). Every BCO $\Sigma$ induces a Set-indexed preorder $[-,\Sigma ]$.

Among BCOs, filtered opcas are characterized as those BCOs for which the Set-indexed preorder $[-, {\cal D}\Sigma ]$ is a tripos. What then might be termed a `complete filtered opca', a BCO $\Sigma$ such that $[-,\Sigma ]$ is a tripos, is characterized (by our proposition~\ref{appl=fpp} and theorem~\ref{Vappmor}) as a ${\cal D}$-algebra whose algebra structure preserves finite meets. This generalizes the characterization of locales among meet-semilattices. We also give a characterization in terms of one of the definitions of Ferrer Santos et al, sharpening their result (theorem~\ref{preimptripos}). Moreover we descrive `dense' morphisms of filtered opcas, and recover a suitable analogue of Peter Johnstone's simple criterion in \cite{JohnstonePT:geomrt}.

Then, we embark on classical realizability. We prove that for every filtered opca $A,A')$ and downwards closed subset $U\subset A$ such that $U\cap A'=\emptyset$, we have an abstract Krivine structure. Moreover, the tripos arising from this aks (by Streicher's construction) represents a topos which is the Booleanization of a closed subtopos of the standard realizability topos ${\sf RT}(A',A)$: that is, for a subobject $U$ of 1 in ${\sf RT}(A',A)$ we get the sheaf subtopos corresponding to the local operator $((-)\to U)\to U$. And, {\em every tripos resulting from an aks is of this form}.

Finally, we investigate when our Boolean triposes are localic. We compare criteria independently given by Hofstra and Krivine, and find them, reassuringly, to be equivalent.

Our final theorem specializes to the filtered pca ${\cal K}_2$ of functions $\mathbb{N}\to\mathbb{N}$ with filter the set of recursive functions. We exhibit a range of non-localic Boolean subtoposes of the Kleene-Vesley topos (theorem~\ref{nonlocalicexmp}).

\section{BCOs, Filtered OPCAs and Triposes}
\subsection{BCOs}\label{BCOsection}
This section rehearses what we shall need from Hofstra's paper \cite{HofstraP:allrr}.
\begin{definition}\label{BCO}\em A {\em Basic Combinatorial Object\/} (BCO) consists of a poset $(\Sigma ,\leq )$ and a set ${\cal F}_{\Sigma}$ of partial endofunctions on $\Sigma$, which structure satisfies the following requirements:\begin{rlist}
\item Every $f\in {\cal F}_{\Sigma}$ has downwards closed domain, and is order-preserving on its domain.
\item There is a total map $i\in {\cal F}_{\Sigma }$ such that $i(a)\leq a$ for all $a\in\Sigma$.
\item For every pair $f,g\in {\cal F}_{\Sigma}$ there is some $h\in {\cal F}_{\Sigma}$ such that whenever $g(f(a))$ is defined, $h(a)\leq g(f(a))$.\end{rlist}\end{definition}
\begin{definition}\label{BCOmor}\em Let $(\Sigma ,\leq ,{\cal F}_{\Sigma})$ and $(\Theta ,\leq ,{\cal F}_{\Theta})$ be BCOs. A {\em morphism\/} from $(\Sigma ,\leq ,{\cal F}_{\Sigma})$  to $(\Theta ,\leq ,{\cal F}_{\Theta})$ is a total function $\phi :\Sigma\to\Theta$ satisfying the conditions:\begin{rlist}
\item There is an element $u\in {\cal F}_{\Theta}$ such that for every inequality $a\leq a'$ in $\Sigma$ we have $u(\phi (a))\leq \phi (a')$ in $\Theta$ (in particular, $u$ is defined on all elements in the image of $\phi$).
\item For every $f\in {\cal F}_{\Sigma}$ there is a $g\in {\cal F}_{\Theta}$ such that for all $a$ in the domain of $f$, $\phi (a)$ is in the domain of $g$, and $g(\phi (a))\leq \phi (f(a))$.\end{rlist}
Given two morphisms $\phi ,\psi :\Sigma\to\Theta$ we say $\phi\leq\psi$ if there is an element $g\in {\cal F}_{\Theta}$ satisfying $g(\phi (a))\leq\psi (a)$ for all $a\in\Sigma$.
\end{definition}
It is readily verified that with these definitions, we have a preorder-enriched category {\bf BCO}. This category has a terminal object 1 and binary products. Therefore, as in any cartesian 2-category, one can speak of objects which have {\em finite internal products}: a BCO has internal terminal object (or: internal top element) if the {\bf BCO}-morphism $\Sigma \to 1$ has a right adjoint (denoted $\top$); and $\Sigma$ has internal binary products (binary meets) if the diagonal map $\Sigma\to\Sigma\times\Sigma$ has a right adjoint. Such a right adjoint, if it exists, will be denoted $(-)\wedge (-)$.

If a BCO has finite internal meets, we define the set ${\rm TV}(\Sigma )$ of {\em designated truth-values\/} as
$${\rm TV}(\Sigma )\; =\;\{ a\in\Sigma\, |\,\text{for some $f\in {\cal F}_{\Sigma}$, $f(\top )\leq a$}\}$$
Clearly, ${\rm TV}(\Sigma )$ is an upwards closed subset of $\Sigma$, and one can show that for all $a,b\in {\rm TV}(\Sigma )$, also $a\wedge b\in {\rm TV}(\Sigma )$. Therefore we think of ${\rm TV}(\Sigma )$ as a {\em filter}. However, bear in mind that $a\wedge b$ is in general {\em not\/} the meet of $a$ and $b$ in the poset $(\Sigma ,\leq )$.
\begin{definition}\label{opca}\em An {\em order-pca\/} or {\em opca\/} is a poset $(A,\leq )$ with a partial binary function (called {\em application}) $A\times A\to A$, written $a,b\mapsto ab$, which has the following properties:\begin{rlist}
\item Whenever $ab$ is defined and $a'\leq a,b'\leq b$, then $a'b'$ is defined and $a'b'\leq ab$.
\item There are elements $\cmb{k}$ and $\cmb{s}$ in $A$ such that for all $x,y\in A$ we have $(\cmb{k}x)y\leq x$, and for all $x,y,z\in A$, whenever $(xz)(yz)$ is defined, so id $((\cmb{s}x)y)z$, and $((\cmb{s}x)y)z\leq (xz)(yz)$.\end{rlist}\end{definition}
From now on, when we work in an order-pca, we associate to the left and write $abc$ instead of $(ab)c$.
\medskip

\noindent Opcas were defined in \cite{HofstraP:ordpca}, and Longley's definition of applicative morphism for pcas (\cite{LongleyJ:reatls}) was extended there to opcas. For more theory of opcas and unexplained notions and notations, the reader is referred to \cite{OostenJ:reaics}.

Every opca $(A,\leq )$ is a BCO $(A,\leq ,{\cal F}_A)$ where ${\cal F}_A$ consists of the partial maps $\phi _a:b\mapsto ab$ given by the opca structure. Moreover, as BCO every opca has finite internal meets (for the map $a\wedge b$ we can take $\cmb{p}ab$, where $\cmb{p}$ is a pairing combinator in $A$; every element of $A$ can serve as top element, and ${\rm TV}(A)=A$).
\begin{theorem}[Hofstra, 5.1]\label{appl=fp} Let $A$ and $B$ be opcas. A function $f:A\to B$ is an applicative morphism of opcas precisely when it is a finite internal meet-preserving morphism of BCOs.\end{theorem}
\begin{definition}\label{filteropca}\em A {\em filtered opca\/} is an opca $A$ together with a subset $A'$ which is closed under application of $A$ and contains choices for $\cmb{k}$ and $\cmb{s}$ (for $A$). The subset $A'$ is called the {\em filter}.\end{definition}
It is immediate that, in definition~\ref{filteropca}, $A'$ itself is an opca; however, not every subset of an opca $A$ which is closed under the application of $A$ and is an opca with this restricted application, is a filter: see \cite{OostenJ:parcaf}, 5.4 for a counterexample.
\medskip

\noindent Every filtered opca $(A,A')$ is a BCO $(A,A',{\cal F}_{A'})$ where ${\cal F}_{A'}$ consists of the partial maps $\phi _a:b\mapsto ab$ with $a\in A'$. Every opca $A$ is of course trivially a filtered opca with $A'=A$. Another example of filtered opcas are meet-semilattices (with top element $\top$): application is the meet operation, and the filter is $\{\top\}$. Many pcas, considered as opcas with the discrete order, contain nontrivial filters: Scott's graph model with the filter of r.e.\ (or, more generally, $\Sigma _n$) subsets of $\mathbb{N}$; Kleene's second model ${\cal K}_2$ of functions $\mathbb{N}\to\mathbb{N}$, with the filter consisting of the total recursive functions (or, more generally, $\Delta _n$-functions).
\medskip

\noindent We need two further notions about BCOs: the {\em downset monad\/} $\cal D$, and the Set-indexed preorder $[-,\Sigma ]$ (for a BCO $\Sigma$).

For any BCO $(\Sigma ,\leq ,{\cal F}_{\Sigma })$ we can consider the poset ${\cal D}\Sigma$ of downwards closed subsets of $\Sigma$, ordered by inclusion, and system of maps ${\cal F}_{{\cal D}\Sigma}$ which consists of those partial maps $F:{\cal D}\Sigma \to {\cal D}\Sigma$ for which there is some $f\in {\cal F}_{\Sigma}$ such that, for all $U\in {\cal D}\Sigma$, $FU$ is defined if and only if $U$ is a subset of the domain of $f$, in which case $FU$ is the downwards closure of $\{ f(a)\, |\, a\in U\}$.

The operation $\cal D$ is the object part of a 2-monad on {\bf BCO}: the unit $\Sigma\to {\cal D}\Sigma$ is given by the principal downset map ${\downarrow}(-)$ sending $a\in\Sigma$ to $\{ b\, |\, b\leq a\}$, and multiplication is union. The monad $\cal D$ is a so-called {\em KZ-monad}; this means that any object carries at most one algebra structure ${\cal D}\Sigma\to\Sigma$, and this structure, when it exists, is left adjoint to the unit map.

We note that if $\Sigma$ is a filtered opca, so is ${\cal D}\Sigma$: if $\Sigma =(A,A')$ then ${\cal D}\Sigma = ({\cal D}A,\Phi )$ where $\Phi$ consists of those downsets of $A$ that meet the filter $A'$.

Every BCO $\Sigma$ gives rise to a Set-indexed preorder $[-,\Sigma ]$: for a set $X$, we have the set of (total) functions $X\to\Sigma$, and for two such functions $\phi ,\psi$ we have $\phi\leq\psi$ if and only if there exists $f\in {\cal F}_{\Sigma}$ such that for all $x\in X$, $f(\phi (x))$ is defined and $f(\phi (x))\leq\psi (x)$. If $\Sigma$ is a filtered opca $(A,A')$, we shall abuse language and write $[-,A]$ for the induced Set-indexed preorder, even though one should be aware that the preorder involves $A'$.

We shall be interested in conditions under which the preorder $[-,\Sigma ]$ is a tripos.

We note that the assignment $\Sigma \mapsto [-,\Sigma ]$ gives a full 2-embedding of {\bf BCO} into the 2-category of Set-indexed preorders. We also note, that the indexed preorder $[-,\Sigma ]$ has indexed finite meets if and only if the BCO $\Sigma$ has internal finite meets. Moreover, a map $h$ between BCOs with internal finite meets preserves those meets if and only if the corresponding transformation between the indexed preorders preserves indexed finite meets.

The following pretty theorem characterizes the filtered opcas among BCOs, in terms of the two notions just discussed:
\begin{theorem}[Hofstra, 6.9]\label{DStriposchar} A BCO $\Sigma$ is (equivalent to) a filtered opca, precisely when the indexed preorder $[-, {\cal D}\Sigma ]$ is a tripos.\end{theorem}
\begin{theorem}[Hofstra, 6.13]\label{Striposchar} Let $\Sigma$ be a BCO with internal finite meets. Then $[-,\Sigma ]$ is a tripos, precisely when $\Sigma$ is a filtered opca $(A,A')$ which has a pseudo-$\cal D$-algebra structure $\bigvee :{\cal D}\Sigma\to \Sigma$ which satisfies the following condition:\begin{itemize}
\item[$(\ast )$] There is an element $v\in A'$ such that whenever we have an $\alpha\in {\cal D}A$ and $b,c\in A$ for which, for each $a\in\alpha$, $ab$ is defined and $\leq c$, then $v(\bigvee\alpha )b$ is defined and $\leq c$.\end{itemize}\end{theorem}
See our theorem~\ref{Vappmor} for a more elegant formulation of the condition $(\ast )$.
\medskip

\noindent We conclude this overview of Hofstra's results with some material on geometric morphisms.
\begin{definition}[Hofstra, 7.1]\label{cdforBCOs}\em A morphism $\phi :\Sigma\to\Theta$ of BCOs is called {\em (computationally) dense\/} if there is some $h\in {\cal F}_{\Theta}$ and a function $H:{\cal F}_{\Theta}\to {\cal F}_{\Sigma}$ satisfying the following property: for $a\in\Sigma$ and $g\in {\cal F}_{\Theta}$, if $\phi (a)$ is in the domain of $g$ then $H(g)(a)$ and $h(\phi (H(g)(a)))$ are defined, and  $h(\phi (H(g)(a)))\leq g(\phi (a))$.\end{definition}
\begin{theorem}[Hofstra, 7.2]\label{cdsubcat} BCOs with dense maps form a sub-preorder-enriched category of {\bf BCO}, to which the monad $\cal D$ restricts.\end{theorem}
\begin{theorem}\label{cdadjoint}\begin{rlist}\item {\bf [Hofstra, 7.3]} For a morphism $\phi$ of BCOs we have: $\phi$ is dense precisely when ${\cal D}\phi$ has a right adjoint. \item {\bf [Hofstra, 7.8]} If $\phi$ is a map between $\cal D$-algebras, then $\phi$ is dense if and only if it has a right adjoint.\end{rlist}\end{theorem}
\begin{theorem}[Hofstra, 7.9]\label{localicchar} Let $\Sigma$ be a BCO such that $[-, \Sigma ]$ is a tripos. Then this tripos is localic if and only if the preorder ${\rm TV}(\Sigma )$ has a least element.\end{theorem}
\subsection{Filtered opcas, triposes and dense morphisms}
In this section we present some notions Hofstra did not explicitly give in his paper. In particular, we need an appropriate definition of morphism between filtered opcas, as well as a characterization of the dense ones among these. Moreover, we have some refinements and generalizations.
\begin{definition}\label{applfilt}\em Let $(A,A')$ , $(B,B')$ be filtered opcas. An {\em applicative morphism\/} $(A,A')\to (B,B')$ is a function $f:A\to B$ satisfying the following conditions:\begin{rlist}\item For all $a\in A'$ there is a $b\in B'$ with $b\leq f(a)$ (so, $f$ maps $A'$ into the upwards closure of $B'$).
\item There is an element $r\in B'$ such that for all $a'\in A'$ and $a\in A$, whenever $a'a$ is defined in $A$, $rf(a')f(a)$ is defined in $B$ and $rf(a')f(a)\leq f(a'a)$.
\item There is an element $u\in B'$ such that for every inequality $x\leq y$ in $A$, $uf(x)$ is defined and $uf(x)\leq f(y)$.\end{rlist}\end{definition}
The following result corresponds to theorem~\ref{appl=fp}:
\begin{proposition}\label{appl=fpp} For filtered opcas $(A,A')$ and $(B,B')$, a function $f:A\to B$ is an applicative morphism precisely when it is a finite-meet preserving map of BCOs.\end{proposition}
\begin{proof} Let $\phi :(A,A')\to (B,B')$ be an applicative morphism. Then $\phi$ is a map of BCOs: requirement i) of definition~\ref{BCOmor} is identical to requirement iii) of \ref{applfilt}, and for ii) of \ref{BCOmor}, given $f\in A'$, pick $b\in B'$ such that $b\leq \phi (f)$ (by i) of \ref{applfilt}) and let $g\equiv \langle y\rangle rby$, where $r$ is from ii) of \ref{applfilt}. If $fa$ is defined in $A$, then $r\phi (f)\phi (a)\leq\phi (fa)$, so $g\phi (a)\leq rb\phi (a)\leq r\phi (f)\phi (a)\leq \phi (fa)$.

We need to show that $\phi$ preserves internal finite meets. Since $\phi$ maps $A'$ into the upwards closure of $B'$, $\phi$ preserves the terminal object. Binary internal meets are given by the pairing combinators in the respective opcas. If we denote pairing and unpairing in $A$ by $\cmb{p},\cmb{p}_0,\cmb{p}_1$ and in $B$ by $\cmb{q},\cmb{q}_0,\cmb{q}_1$, then for
$$\cmb{t}\; =\; \langle x\rangle \cmb{q}(r\phi (\cmb{p}_0)x)(r\phi (\cmb{p}_1)x)$$
we have $\cmb{t}\phi (\cmb{p}aa')\leq\cmb{q}\phi (a)\phi (a')$ for all $a,a'\in A$; for
$$\cmb{u}\; =\;\langle x\rangle r(r\phi (\cmb{p})(\cmb{q}_0x))(\cmb{q}_1x)$$
we have $\cmb{u}(\cmb{q}\phi (a)\phi (a'))\leq\phi (\cmb{p}aa')$ for all $a,a'\in A$. So $\phi$ preserves internal finite meets.

Conversely, suppose $\phi :(A,A')\to (B,B')$ is a morphism of BCOs which preserves internal finite meets. Requirement i) of \ref{applfilt} is satisfied because $\phi$ preserves top elements. Requirement iii) is satisfied because $\phi$ is a map of BCOs. As for requirement ii), let $\alpha\in B'$ be such that for all $a,a'\in A$,
$$\alpha (\cmb{q}\phi (a)\phi (a'))\leq\phi (\cmb{p}aa')$$
(since $\phi$ preserves finite meets). There is an element $d\in A'$ such that whenever $aa'$ is defined in $A$, $d(\cmb{p}aa')\leq aa'$. Since $\phi$ is a map of BCOs, there is $e\in B'$ such that when $aa'$ is defined in $A$, $e\phi (\cmb{p}aa')\leq\phi (aa')$. Combining, we have for $a,a'\in A$ such that $aa'$ is defined,
$$e\alpha (\cmb{q}\phi (a)\phi (a'))\leq\phi (aa')$$
so if $r=\langle xy\rangle e\alpha (\cmb{q}xy)$ then $r$ satisfies requirement ii) of \ref{applfilt}.\end{proof}
\medskip

\noindent Next, we look at filtered opcas $(A,A')$ for which the indexed preorder $[-,A]$ is a tripos. By Hofstra's theorem~\ref{Striposchar}, $(A,A')$ carries the structure of a pseudo $\cal D$-algebra satisfying the condition $(\ast )$. In order to be explicit and to fix notation, let us define what we mean by ``pseudo $\cal D$-algebra'':
\begin{definition}\label{Vchar}\em A pseudo $\cal D$-algebra structure on a BCO $\Sigma$ is a function $\bigvee :{\cal D}\Sigma\to\Sigma$ satisfying the following conditions, where we write ${\downarrow}\alpha$ for the downwards closure of $\alpha$, and ${\downarrow}a$ for ${\downarrow}\{ a\}$:\begin{arlist}
\item There is $u\in {\cal F}_{\Sigma}$ such that for every inclusion $\alpha\subseteq\alpha '$ in ${\cal D}{\Sigma}$, $u(\bigvee\alpha )$ is defined and $u(\bigvee\alpha )\leq\bigvee\alpha '$.
\item For all $f\in {\cal F}_{\Sigma}$ there is some $g_2\in {\cal F}_{\Sigma}$ such that for all $\alpha\in {\cal D}\Sigma$: if for all $x\in\alpha$ $f(x)$ is defined, then $g_2(\bigvee\alpha )$ is defined and $g_2(\bigvee\alpha )\leq\bigvee ({\downarrow}\{ f(x)\, |\, x\in\alpha\} )$.
\item There are elements $g_3,h_3\in {\cal F}_{\Sigma}$ such that for all ${\cal A}\in {\cal D}^2\Sigma$:
$$\begin{array}{l} g_3(\bigvee ({\downarrow}\{\bigvee\alpha\, |\,\alpha\in {\cal A}\} ))\leq\bigvee (\bigcup {\cal A}) \\
h_3(\bigvee (\bigcup {\cal A}))\leq\bigvee ({\downarrow}\{\bigvee\alpha\, |\,\alpha\in {\cal A}\} ) \end{array}$$
\item There are elements $g_4,h_4\in {\cal F}_{\Sigma}$ such that for all $a\in\Sigma$, $g_4(\bigvee ({\downarrow}a))\leq a$ and $h_4(a)\leq\bigvee ({\downarrow}a)$.
\end{arlist}\end{definition}
\begin{theorem}\label{Vappmor} A filtered opca $(A,A')$ with pseudo ${\cal D}$-algebra structure $\bigvee$ satisfies condition $(\ast )$ of \ref{Striposchar}, precisely when $\bigvee$ is an applicative morphism of filtered opcas.\end{theorem}
\begin{proof} First suppose $\bigvee$ is an applicative morphism. So, we have $r\in A'$ such that whenever $\alpha\beta$ is defined in ${\cal D}(A,A')$, $r(\bigvee\alpha )(\bigvee\beta )$ is defined and $r(\bigvee\alpha )(\bigvee\beta )\leq\bigvee\dar\{ ab\, |\, a\in\alpha ,b\in\beta\}$.

Now suppose for all $a\in\alpha$ that $ab$ is defined and $ab\leq c$. Then $\alpha (\dar b)$ is defined and $$\alpha (\dar b)\subseteq\dar\{ ab\, |\, a\in\alpha ,b\in\beta\}\subseteq\dar c$$
We have $r(\bigvee\alpha )(\bigvee\dar b)\leq\bigvee (\alpha (\dar b))$. We have $u\in A'$ such that
$$u(r(\bigvee\alpha )(\bigvee\dar b))\leq u(\bigvee (\alpha (\dar b)))\leq\bigvee\dar c$$
Let $g_4, h_4\in A'$ be as given by definition~\ref{Vchar} 4). Then since $h_4b\leq\bigvee\dar b$,
$$\begin{array}{l} g_4(u(r(\bigvee\alpha )(\bigvee\dar b)))\leq g_4(\bigvee\dar c)\leq c \\
g_4(u(r(\bigvee\alpha )(h_4 b)))\leq g_4(\bigvee\dar c)\leq c\end{array}$$
Now let $\cmb{v} =\langle xy\rangle g_4(u(rx(h_4y)))$. It is easy to verify that $v\in A'$ and that $v$ satisfies the condition $(\ast )$.
\medskip

\noindent Conversely, suppose $v$ satisfies $(\ast )$. We have to prove that $\bigvee$ is an applicative morphism. For i) of \ref{applfilt}, we have to prove that for $\alpha$ in the filter of ${\cal D}(A,A')$, $\bigvee\alpha$ is in the upwards closure of $A'$. The filter of ${\cal D}(A,A')$ consists of those downsets of $A$ which intersect $A'$. Pick $a\in\alpha \cap A'$. Then $\dar a\subseteq\alpha$ so $u(\bigvee \dar a)\leq\bigvee\alpha$ (where $u$ is from 1) of \ref{Vchar}). And $h_4a\leq\bigvee\dar a$, so
$$u(h_4a)\leq u(\bigvee\dar a)\leq\bigvee\alpha$$
Since $u$ and $h_4$ are in $A'$, we see that i) is satisfied.

Condition iii) of \ref{applfilt} holds because $\bigvee$ is supposed to be a map of BCOs.

For \ref{applfilt} ii), suppose $\alpha\beta$ is defined, so for all $a\in\alpha ,b\in\beta$, $ab$ is defined in $A$. Note that $u(\bigvee\dar ab)\leq\bigvee\alpha\beta$ for $a\in\alpha ,b\in\beta$. Also, $h_4(ab)\leq\bigvee\dar (ab)$, so
$$u(h_4(ab))\leq u(\bigvee\dar (ab))\leq\bigvee (\alpha\beta )$$
Let $\xi =\langle xy\rangle u(h_4(xy))$, then $\xi\in A'$ and for a fixed $b\in\beta$ we have for all $a\in\alpha$, $\xi ab\leq\bigvee (\alpha\beta )$. By $(\ast )$ we have that $v(\bigvee\{\xi a\, |\, a\in\alpha\} )b$ is defined and
$$v(\bigvee\{\xi a\, |\, a\in\alpha\} )b\leq\bigvee (\alpha\beta )$$
Let $\eta =\langle yx\rangle vxy$. Then $\eta\in A'$ and
$$\eta b(\bigvee\{ \xi a\, |\, a\in\alpha\} )\leq v(\bigvee\{\xi a\, |\, a\in\alpha\} )b\leq\bigvee (\alpha\beta )$$
This holds for all $b\in\beta$, so by $(\ast )$ we have
$$v(\bigvee\{\eta b\, |\, b\in\beta\} )(\bigvee\{\xi a\, |\, a\in\alpha\} )\leq\bigvee (\alpha\beta )$$
Hence,
$$\eta (\bigvee\{\xi a\, |\, a\in\alpha\} )(\bigvee\{\eta b\, |\, b\in\beta\} )\leq\bigvee (\alpha\beta )$$
By 2) of \ref{Vchar}, choose $\xi ',\eta '\in A'$ such that for all $\alpha ,\beta$,
$$\begin{array}{l} \xi '(\bigvee\alpha )\leq\bigvee\{\xi a\, |\, a\in\alpha\} \\
\eta '(\bigvee\beta )\leq\bigvee\{\eta b\, |\, b\in\beta \} \end{array}$$
and let $z=\langle xy\rangle \eta (\xi 'x)(\eta 'y)$. Then $z\in A'$ and
$$\begin{array}{lllll} z(\bigvee\alpha )(\bigvee\beta ) & \leq & \eta (\xi '(\bigvee\alpha ))(\eta '(\bigvee\beta )) & & \\
 & \leq & \eta (\bigvee\{\xi a\, |\, a\in\alpha\} )(\bigvee\{\eta b\, |\, b\in\beta\} ) & \leq & \bigvee (\alpha\beta ) \end{array}$$
so $z$ realizes condition ii) of \ref{applfilt}.\end{proof}

\begin{remark}\em 1. Note that one may reformulate theorem~\ref{Vappmor} thus: for a filtered opca $(A,A')$, the Set-indexed preorder $[-,A]$ is a tripos precisely when $(A,A')$ is a pseudo $\cal D$-algebra {\em in the subcategory (of BCO) of filtered opcas and applicative morphisms}.

2. Theorem~\ref{Vappmor} is, in view of proposition~\ref{appl=fpp}, the generalization to the context of filtered opcas, of the condition of infinite distributivity for locales. Indeed, a suplattice $L$ is a locale precisely when the supremum map $\bigvee :{\cal D}L\to L$ preserves finite meets: $\bigvee (\alpha\cap\beta )=(\bigvee\alpha )\wedge (\bigvee\beta )$.\end{remark}

\noindent Let us draw an immediate inference from theorem~\ref{Vappmor}:
\begin{corollary}\label{SsubtriposDS} Suppose $\Sigma$ is a BCO such that $[-,\Sigma ]$ is a tripos. Then $[-,\Sigma ]$ is a subtripos of $[-,{\cal D}\Sigma ]$.\end{corollary}
\begin{proof} The assumption implies, by \ref{Striposchar}, that $\Sigma$ is a filtered opca, and, by \ref{Vappmor} and \ref{appl=fpp}, that the transformation of indexed preorders $[-,{\cal D}\Sigma ]\to [-,\Sigma ]$ induced by $\bigvee$, preserves finite meets. Hence the pair $\bigvee\dashv \dar (-)$ defines a geometric inclusion of $[-,\Sigma ]$ into $[-, {\cal D}\Sigma ]$.\end{proof}
\medskip

\noindent At this point we would like to relate our notion of filtered opcas satisfying the condition of \ref{Striposchar}, to the notion of {\em implicative oca\/} discussed in \cite{FerrerSantosW:ordcar}.
\begin{definition}[Ferrer Santos et al]\label{ioca}\em An {\em implicative ordered combinatory algebra\/} (ioca) is a filtered opca $(A,A')$ satisfying the following conditions:\begin{alist}
\item The application map is total.
\item The poset $(A,\leq )$ has infima of arbitrary subsets.
\item There is an operation called {\em implication}, $a,b\mapsto (a\Rightarrow b)$, order-reversing in the first argument and order-preserving in the second, such that for all $a,b,c\in A$ we have: $a\leq (b\Rightarrow c)$ if and only if $ab\leq c$.
\item There is an element $\cmb{e}\in A'$ such that for all $a,b,c\in A$: if $ab\leq c$ then $\cmb{e}a\leq (b\Rightarrow c)$.\end{alist}\end{definition}
Ferrer Santos et al prove the following result:
\begin{theorem}[Ferrer Santos et al, 5.8]\label{iocathm} If $(A,A')$ is an implicative oca, then $[-,A]$ is a tripos.\end{theorem}
However, the requirements for an ioca are too strong for the conclusion. We reformulate the notion so that we obtain an equivalence.
\begin{definition}\label{preiopca}\em A {\em pre-implicative opca\/} is a filtered opca $(A,A')$ satisfying the following conditions:\begin{rlist}
\item There is a map $\bigwedge :{\cal P}(A)\to A$ and there are constants $\cmb{i},\cmb{i}'\in A'$ such that for all $\alpha\subseteq A$:
$$\begin{array}{l}\text{for all }a\in\alpha ,\; \cmb{i}(\bigwedge\alpha )\leq a \\
\text{for all $b\in A$, if }b\leq a\text{ for all $a\in\alpha$, then }\cmb{i}'b\leq\bigwedge\alpha \end{array}$$
\item There is a binary implication $a,b\mapsto (a\Rightarrow b)$ on $A$ and there are constants $\cmb{e},\cmb{e}'\in A'$ satisfying, for all $a,b,c\in A$:
$$\begin{array}{l}\text{if $ab\leq c$ then }\cmb{e}a\leq (b\Rightarrow c) \\
\text{if $a\leq (b\Rightarrow c)$ then }\cmb{e}'ab\leq c \end{array}$$\end{rlist}
Note that in particular, the application on $A$ need not be total.\end{definition}
\begin{theorem}\label{preimptripos} Let $\Sigma$ be a BCO. Then $[-,\Sigma]$ is a tripos if and only if $\Sigma$ is a pre-implicative opca.\end{theorem}
\begin{proof}First, suppose that $[-,\Sigma ]$ is a tripos. By Hofstra's theorem \ref{Striposchar}, we know that $\Sigma$ is a filtered opca $(A,A')$ which carries a pseudo $\cal D$-algebra structure $\bigvee :{\cal D}A\to A$, which satisfies condition $(\ast )$.

For $\alpha ,\beta\in {\cal D}A$ we define $I(\alpha ,\beta )$ as
$$I(\alpha ,\beta )\; =\; \{ a\in A\, |\, \text{for all $a'\in \alpha$, $aa'$ is defined and an element of $\beta$}\}$$
Clearly, $I(\alpha ,\beta )\in {\cal D}A$. Define the operation $\Rightarrow$ by
$$(b\Rightarrow c)\; =\; \bigvee I(\dar b,\dar c)$$
Now if $ab\leq c$ then clearly $a\in I(\dar b,\dar c)$, so $\dar a\subseteq I(\dar b,\dar c)$ so with $u$ as in 1) of \ref{Vchar} we have $u(\bigvee \dar a)\leq (b\Rightarrow c)$ and since with $h_4$ from 4) in \ref{Vchar} we have $h_4a\leq\bigvee\dar a$, we find 
$$u(h_4a)\leq (b\Rightarrow c)$$
So if $\cmb{e}$ is defined as $\langle x\rangle u(h_4x)$ then $\cmb{e}$ satisfies the first condition in ii) of definition~\ref{preiopca}.

For the second condition of \ref{preiopca} ii), we note that for $a\in I(b,c)$ we have $ab$ defined and $ab\leq c$; by $(\ast )$ we see that $v(b\Rightarrow c)b$ is defined and $\leq c$; so if $a\leq (b\Rightarrow c)$ then $vab\leq c$. Hence we can take $v$ as our $\cmb{e}'$, and we conclude that the operation $\Rightarrow$ and the constants $\cmb{e},\cmb{e}'$ satisfy \ref{preiopca} ii).

For the map $\bigwedge$, defined on arbitrary subsets $\alpha\subseteq A$, let $O(\alpha )$ be the set of lower bounds of $\alpha$ (then $O(\alpha )\in {\cal D}A$) and put 
$$\bigwedge \alpha\; =\; \bigvee O(\alpha )$$
If $a\in\alpha$ is arbitrary, then for all $b\in O(\alpha )$, $\cmb{skk}b\leq a$. So if $g_2$ is as in \ref{Vchar} 2) for $\cmb{skk}$, then $g_2(\bigwedge\alpha )\leq a$. Hence we can take $g_2$ as our $\cmb{i}$.

The second condition reads: for all $b\in O(\alpha )$, $\cmb{i}'b\leq\bigwedge\alpha =\bigvee O(\alpha )$. But we have a combinator $w\in A'$ such that whenever $\beta\in {\cal D}A$ and $a\in\beta$, $wa\leq\bigvee\beta$. So it is clear how to pick $\cmb{i}'$. We conclude that $(A,A')$ has the structure of a pre-implicative opca.
\medskip

\noindent Conversely, suppose $(A,A')$ is a filtered opca endowed with operations $\bigwedge$ and $\Rightarrow$ and elements $\cmb{i},\cmb{i}',\cmb{e},\cmb{e}'$ satisfying the conditions for a pre-implicative opca. For an indexed family $\{\Phi x\, |\, x\in X\}$ of elements of $A$, we shall also write $\bigwedge _{x\in X}\Phi x$ for $\bigwedge\{ \Phi x\, |\, x\in X\}$.

Define $\bigvee :{\cal D}A\to A$ by
$$\bigvee\alpha\; =\; \bigwedge_{b\in A}(\bigwedge_{a\in\alpha}(a\Rightarrow b)\Rightarrow b)$$
We prove that this map $\bigvee$ provides $(A,A')$ with a pseudo $\cal D$-algebra structure which satisfies condition $(\ast )$ of theorem~\ref{Striposchar}.

We define a number of elements of $A'$:
$$\begin{array}{rcl}\eta & = & \langle x\rangle\cmb{i}'(\cmb{i}x) \\
\xi & = & \langle x\rangle\cmb{e}(\cmb{e}'x) \\
H & = & \langle xy\rangle\cmb{e}'(\xi x)(\eta y) \\
K  & = & \langle x\rangle\cmb{i}'(\cmb{e}(H(\cmb{i}x))) \\
P & = & \langle uv\rangle\cmb{e}'(\cmb{i}v)(\cmb{i}'(\cmb{e}u)) \end{array}$$
And we note the following facts concerning these elements:\begin{alist}
\item For $\alpha\in {\cal D}A$, a family $\{\Phi _a\, |\, a\in\alpha\}$ and $\alpha '\subseteq\alpha$, we have 
$$\eta (\bigwedge_{a\in\alpha}\Phi _a)\leq\bigwedge_{a\in\alpha '}\Phi _a$$
\item For inequalities $b\leq b'$, $c\leq c'$ in $A$ we have
$$\begin{array}{rcl} \xi (b'\Rightarrow c) & \leq & (b\Rightarrow c) \\
\xi (b\Rightarrow c) & \leq & (b\Rightarrow c')\end{array}$$
\item For $\alpha\subseteq\alpha '$ in ${\cal D}A$ we have
$$K(\bigvee \alpha )\leq\bigvee\alpha '$$
\item If $f\in A$, $\alpha\in {\cal D}A$ and $fa\leq b$ for every $a\in\alpha$, then
$$Pf(\bigvee\alpha )\leq b$$\end{alist}
By way of example, we spell out the proof of c); the proof of the other statements is left to the reader. Assume $\alpha\subseteq\alpha '$. Then by a), $\eta (\bigwedge_{a\in\alpha '}(a\Rightarrow b))\leq\bigwedge_{a\in\alpha}(a\Rightarrow b)$, so by b), $\xi (\bigwedge_{a\in\alpha}(a\Rightarrow b)\Rightarrow b)\leq\eta (\bigwedge_{a\in\alpha '}(a\Rightarrow b))\Rightarrow b$, hence $\cmb{e}'(\xi (\bigwedge_{a\in\alpha}(a\Rightarrow b)\Rightarrow b))(\eta (\bigwedge_{a\in\alpha '}(a\Rightarrow b)))\leq b$. Therefore, by definition of $H$, $H(\bigwedge_{a\in\alpha}(a\Rightarrow b)\Rightarrow b)(\bigwedge_{a\in\alpha '}(a\Rightarrow b))\leq b$, hence $\cmb{e}(H(\bigwedge_{a\in\alpha }(a\Rightarrow b)\Rightarrow b))\leq\bigwedge_{a\in\alpha '}(a\Rightarrow b)\Rightarrow b$. Since $\cmb{i}(\bigvee\alpha )\leq\bigwedge_{a\in\alpha}(a\Rightarrow b)\Rightarrow b$ and application is downwards closed and order-preserving, we get $\cmb{e}(H(\cmb{i}(\bigvee\alpha)))\leq\bigwedge_{a\in\alpha '}(a\Rightarrow b)\Rightarrow b$, so $\cmb{i}'(\cmb{e}(H(\cmb{i}(\bigvee\alpha ))))\leq\bigvee\alpha '$, which is $K(\bigvee\alpha )\leq\bigvee\alpha '$ as desired.

Now to prove that $\bigvee$ is a pseudo $\cal D$-algebra map, requirement 1) of definition~\ref{Vchar} follows at once from property c). For requirement 4) we define the element
$$Q\; =\;\langle x\rangle\cmb{i}'(\cmb{e}(\langle uv\rangle\cmb{e}'(\cmb{i}v)u)x)$$
and we claim that whenever $a\in\alpha$, $Qa\leq\bigvee\alpha$. Indeed,
$$\begin{array}{c}\cmb{i}(\bigwedge_{a'\in\alpha}(a'\Rightarrow b))\leq (a\Rightarrow b)\text{, so }\cmb{e}'(\cmb{i}(\bigwedge_{a'\in\alpha}(a'\Rightarrow b)))a\leq b\text{, hence} \\
(\langle uv\rangle\cmb{e}'(\cmb{i}v)u)a(\bigwedge_{a'\in\alpha}(a'\Rightarrow b))\leq b\text{, i.e. } \\
\cmb{e}(\langle uv\rangle\cmb{e}'(\cmb{i}v)u)a\leq\bigwedge_{a'\in\alpha}(a'\Rightarrow b)\Rightarrow b\text{, hence} \\
\cmb{i}'(\cmb{e}(\langle uv\rangle\cmb{e}'(\cmb{i}v)u)a)\leq\bigvee\alpha\text{, which is}\\
Qa\leq\bigvee\alpha\end{array}$$
For the other inequality of \ref{Vchar} 4), we claim that for $R=\langle x\rangle\cmb{e}'(\cmb{i}x)(\cmb{i}'(\cmb{e}(\cmb{skk})))$ we have $R(\bigvee\dar a)\leq a$; the verification is easy.

For requirement 2) we use statement d). Suppose $fa$ is defined for all $a\in\alpha$. Then $Q(fa)\leq\bigvee\dar\{ fa\, |\, a\in\alpha\}$ for all $a\in\alpha$, hence $P(\langle x\rangle Q(fx))(\bigvee\alpha )\leq\bigvee\dar\{ fa\, |\, a\in\alpha\}$ as desired.

For requirement 3), let ${\cal A}\in {\cal D}^2A$, $\alpha\in {\cal A}$. Since $\alpha\subseteq\bigcup {\cal A}$ we have $K(\bigvee\alpha )\leq\bigvee (\bigcup {\cal A})$ by c). Hence $Kx\leq\bigvee (\bigcup {\cal A})$ for all $x\in\dar\{\bigvee\alpha\, |\,\alpha\in {\cal A}\}$, so $PK(\bigvee\dar\{\bigvee\alpha\, |\,\alpha\in {\cal A}\}\leq\bigvee (\bigcup {\cal A})$. The other inequality of 3) is realized by the element $P(\langle x\rangle Q(Qx))$: for $a\in\bigcup {\cal A}$, there is $\alpha\in {\cal A}$ such that $a\in\alpha$. Then $Qa\leq\bigvee\alpha$, so $Qa\in\dar\{\bigvee\alpha\, |\,\alpha \in {\cal A}\}$ whence $Q(Qa)\leq\bigvee\dar\{\bigvee\alpha\, |\,\alpha\in {\cal A}\}$. By d), we have $P(\langle x\rangle Q(Qx))(\bigvee\bigcup {\cal A})\leq\bigvee\dar\{\bigvee\alpha\, |\,\alpha\in {\cal A}\}$. We conclude that $\bigvee $ is a pseudo $\cal D$-algebra structure on $(A,A')$.
\medskip

\noindent It remains to show that the map $\bigvee$ satisfies condition $(\ast )$ of \ref{Striposchar}. This also readily follows from statement d) above. Suppose for all $a\in\alpha$, $ab$ is defined and $ab\leq c$. Then for all $a\in\alpha$, $(\langle x\rangle xb)a\leq c$, whence $P(\langle x\rangle xb)(\bigvee\alpha )\leq c$. Hence if $v=\langle uw\rangle P(\langle x\rangle xw)u$, then $v(\bigvee\alpha )b\leq c$ as required. And obviously, $v\in A'$.
\end{proof}
\medskip

\noindent We now turn to computationally dense maps between filtered opcas. The following definition is a direct translation of Hofstra's general notion of a dense map between BCOs (\ref{cdforBCOs}).
\begin{definition}\label{cdforfilopca}\em Suppose $f:(A,A')\to (B,B')$ is an applicative morphism of filtered opcas. Then $f$ is called {\em computationally dense\/} if there is an element $m\in B'$ with the following property:
$${\rm (cd)}\;\;\;\begin{array}{l} \text{For every $b'\in B'$ there is an $a'\in A'$ such that for all $a\in A$,} \\
\text{if $b'f(a)$ is defined then so is $a'a$, and $mf(a'a)\leq b'f(a)$} \end{array}$$
In Skolemized form, condition (cd) reads:
$${\rm (cd-sk)}\;\;\;\begin{array}{l} \text{There is a function $g:B'\to A'$ such that for all $b'\in B'$ and $a\in A$,} \\
\text{if $b'f(a)$ is defined then so is $g(b')a$, and $mf(g(b')a)\leq b'f(a)$}\end{array}$$\end{definition} 
Peter Johnstone, in \cite{JohnstonePT:geomrt}, has given a simplification of the definition of a computationally dense applicative morphism for pcas. A similar simplification can also be obtained here:
\begin{proposition}\label{ptjdense} Let $f:(A,A')\to (B,B')$ be an applicative morphism. Then $f$ is computationally dense if and only if there is a function $h:B'\to A'$ and an element $\cmb{t}\in B'$ such that for all $b'\in B'$, $$\cmb{t}f(h(b'))\leq b'$$\end{proposition}
\begin{proof} Suppose $f$ is applicative, with elements $r,u\in B'$ satisfying ii) and iii) of definition~\ref{applfilt}, respectively.

For the `only if' part, assume $g:B'\to A'$ and $m\in B'$ satisfy (cd-sk). Pick $a'\in A'$ arbitrary, and fix some $v\in B'$ with $v\leq f(a')$ (by i) of \ref{applfilt}). Define $h(b')=g(\cmb{k}b')$, then $h$ maps $B'$ into $A'$. Let $\cmb{t}=\langle x\rangle m(rxv)$, then $\cmb{t}\in B'$. 

Now for an arbitrary $b'\in B'$, we have $\cmb{k}b'f(a')$ defined, so by (cd-sk) we have $mf(g(\cmb{k}b')a')\leq\cmb{k}b'f(a')\leq b'$. In other words, $mf(h(b')a')\leq b'$. Since in particular $h(b')a'$ is defined, we have $rf(h(b'))f(a')$ defined and $rf(h(b'))f(a')\leq f(h(b')a')$, so $m(rf(h(b'))v)$ is defined and $m(rf(h(b'))v)\leq mf(h(b')a')\leq b'$. We see that $\cmb{t}f(h(b'))$ is defined and $\cmb{t}f(h(b'))\leq b'$, as desired.
\medskip

\noindent For the `if' part, assume $h:B'\to A'$ and $\cmb{t}\in B'$ satisfy $\cmb{t}f(h(b'))\leq b'$ for all $b'\in B'$. Let $\cmb{p},\cmb{p}_0,\cmb{p}_1$ be pairing and unpairing operators in $A'$. Choose $\cmb{q}_0,\cmb{q}_1\in B'$ with $\cmb{q}_i\leq f(\cmb{p}_i)$ (by \ref{applfilt} i)). Suppose $b'\in B'$, $b'f(a)$ defined. Then $\cmb{t}f(h(b'))f(a)\leq b'f(a)$. Since $\cmb{p}_0(\cmb{p}h(b')a)\leq h(b')$ we have
$$rf(\cmb{p}_0)f(\cmb{p}h(b')a)\leq f(\cmb{p}_0(\cmb{p}h(b')a))$$ and hence
$$u(r\cmb{q}_0f(\cmb{p}h(b')a))\leq uf(\cmb{p}_0(\cmb{p}h(b')a))\leq f(h(b'))$$
Let $N=\langle x\rangle u(r\cmb{q}_0x)$, so $Nf(\cmb{p}h(b')a)\leq f(h(b'))$. Then
$$\cmb{t}(Nf(\cmb{p}h(b')a))f(a)\leq \cmb{t}f(h(b'))f(a)\leq b'f(a)$$
Also, since $\cmb{p}_1(\cmb{p}h(b')a)\leq a$ we have, in a similar way,
$$u(r\cmb{q}_1f(\cmb{p}h(b')a))\leq uf(\cmb{p}_1(\cmb{p}h(b')a))\leq f(a)$$
Let $M=\langle x\rangle u(r\cmb{q}_1x)$. We see that for $\cmb{m}=\langle x\rangle\cmb{t}(Nx)(Mx)$, we have
$$\cmb{m}f(\cmb{p}h(b')a)\leq b'f(a)$$
So if we define $g(b')$ as $\cmb{p}h(b')$ then $\cmb{m}f(g(b')a)\leq b'f(a)$, as desired.\end{proof}
\medskip

\noindent Now suppose $\Sigma$ and $\Theta$ are BCOs such that $[-,\Sigma ]$ and $[-,\Theta ]$ are triposes. Then by \ref{Striposchar}, $\Sigma$ and $\Theta$ are filtered opcas which are also pseudi ${\cal D}$-algebras.

Every geometric morphism $[-,\Sigma ]\to [-,\Theta ]$ arises (by fullness of the embedding of BCO into the 2-category of Set-indexed preorders) from an adjoint pair of maps between $\Sigma$ and $\Theta$ which preserve internal finite meets; that is, by \ref{appl=fpp}, an adjoint pair of applicative morphisms. Since a map between $\cal D$-algebras is dense precisely when it has a right adjoint, we see that such geometric morphisms are uniquely determined by computationally dense applicative morphisms $\Theta\to\Sigma$.

\section{Krivine structures and triposes, and filtered opcas}
Thomas Streicher (\cite{StreicherT:kcrcp}) has reformulated Krivine's classical realizability (as presented in, e.g., \cite{KrivineJL:typlcc,KrivineJL:depcqc}) in a style reminiscent of combinatory logic, and therefore susceptible to an analysis with notions from the theory of pcas. He formulates the notion of an {\em abstract Krivine structure}. Out of an abstract Krivine structure one constructs a filtered opca $\Sigma$ (in fact, an implicative oca in the terminology of Ferrer Santos et al--see\ref{ioca}) such that the tripos $[-, \Sigma ]$ is Boolean. This provides a link between Krivine's interpretations of Set Theory and Topos Theory. It is an interesting question whether in the topos resulting from $[-,\Sigma ]$ one can build (using the ideas of algebraic set theory, for which see \cite{JoyalA:algst}) {\em internal models\/} which would faithfully reflect Krivine's interpretations; as was done, for example, in Hyland's effective topos, for the Friedman-McCarty realizability interpretation for IZF, in \cite{KouwenhovenC:algste}.

The first author discovered that, given a filtered opca $(A,A')$ and a nontrivial subterminal object in the relative realizability topos ${\sf RT}(A,A')$, one can construct an abstract Krivine structure (\cite{OostenJ:clar}). A similar idea appeared in Wouter Stekelenburg's PhD thesis (\cite{StekelenburgW:reac}). This section provides the details and also shows that, up to equivalence of the resulting toposes, {\em every abstract Krivine structure arises in this way}.

This means we have a pretty concrete way to present toposes arising out of abstract Krivine structures; but we still have to filter out the non-localic toposes. These are the ones of interest, as the set theory of Boolean localic toposes is basically forcing (see \cite{BellJL:sett} for an exposition). It turns out that Hofstra's condition \ref{localicchar} (which we shall compare with a criterion given by Krivine) gives rise to some recursion-theoretic calculations in our pet example: Kleene's second model of functions $\mathbb{N}\to\mathbb{N}$, with the total recursive functions as filter.
\begin{definition}[Streicher]\label{aksdef}\em An {\em abstract Krivine structure\/} (aks) consists of the following data:\begin{rlist}
\item A set $\Lambda$ of {\em terms}, together with a binary operation $t,s\mapsto t{\cdot}s: \Lambda\times\Lambda\to\Lambda$, and distinguished elements $\cmb{K},\cmb{S},\cc$.
\item A subset {\sf QP} of $\Lambda$ (the set of {\em quasi-proofs}), which contains $\cmb{K},\cmb{S}$ and $\cc$, and is closed under the binary operation of i).
\item A set $\Pi$ of {\em stacks\/} together with a `push' operation
$$t,\pi\mapsto t.\pi: \Lambda\times \Pi\to\Pi$$
(when we iterate this operation, we associate to the right, and write $t.s.\pi$ for $t.(s.\pi )$), as well as an operation
$$\pi\mapsto k_{\pi}: \Pi\to\Lambda$$
\item A subset $\pole$ (the {\em pole\/}) of $\Lambda\times\Pi$, which satisfies the following requirements:\begin{itemize}
\item[(S1)] If $(t,s.\pi )\in\pole$ then $(t{\cdot}s,\pi )\in\pole$
\item[(S2)] If $(t,\pi )\in\pole$ then $(\cmb{K},t.s.\pi )\in\pole$ (for any term $s$)
\item[(S3)] If $((t{\cdot}u){\cdot}(s{\cdot}u),\pi )\in\pole$ then $(\cmb{S},t.s.u.\pi )\in\pole$
\item[(S4)] If $(t,k_{\pi}.\pi )\in\pole$ then $(\cc ,t.\pi )\in\pole $
\item[(S5)] If $(t,\pi )\in\pole$ then $(k_{\pi},t.\pi ')\in\pole$ (for any $\pi '$)\end{itemize}\end{rlist}\end{definition}
Given a set $U$ of terms and a set $\alpha$ of stacks, we define
$$\begin{array}{rcl} U ^{\pole} & = & \{\pi\in\Pi\, |\,\text{for all $t\in U$, }(t,\pi )\in\pole\} \\
\alpha ^{\pole} & = & \{ t\in\Lambda\, |\,\text{for all $\pi\in\alpha$, }(t,\pi )\in\pole\} \end{array}$$
Clearly, we have closure operators $(-)^{\pole\pole}$ on both ${\cal P}(\Lambda )$ and ${\cal P}(\Pi )$. For $\alpha\subseteq\Pi$, we also write $|\alpha |$ for $\alpha ^{\pole}$.

Let ${\cal P}_{\pole}(\Pi )$ be $\{\beta\subseteq\Pi\, |\,\beta ^{\pole\pole}=\beta\}$, ordered by {\em reverse\/} inclusion. We define an application $\bullet$ on ${\cal P}_{\pole}(\Pi )$ by putting
$$\alpha {\bullet}\beta\; =\; \{\pi\in \Pi\, |\,\text{for all $t\in |\alpha |$ and $s\in |\beta |$, }(t,s.\pi )\in\pole\} ^{\pole\pole}$$
Moreover, let $\Phi\subseteq {\cal P}_{\pole}(\Pi )$ be the set
$$\Phi\; =\; \{\alpha\in {\cal P}_{\pole}(\Pi )\, |\, |\alpha |\cap {\sf QP}\neq\emptyset\}$$
\begin{theorem}[Streicher]\label{aks=>opca} The set ${\cal P}_{\pole}(\Pi )$ forms, together with the given application, a total order-ca, and $\Phi$ is a filter in it. The Set-indexed preorder $[-,{\cal P}_{\pole}(\Pi )]$ is a Boolean tripos.\end{theorem}
Ferrer Santos et al (\cite{FerrerSantosW:ordcar}) observe that, in fact, the order-ca ${\cal P}_{\pole}(\Pi )$ is an implicative order-ca (see definition~\ref{ioca}), with implication defined by
$$\alpha\Rightarrow\beta\; =\;\{ t.\pi\, |\, t\in |\alpha |,\pi\in\beta\}^{\pole\pole}$$
and that the element $\{\cc\} ^{\pole}$ realizes `Pierce's Law':
$$\{\cc\} ^{\pole}\leq ((\alpha\Rightarrow\beta )\Rightarrow\alpha )\Rightarrow\alpha$$
Consequently, the define a {\em Krivine order-ca\/} as an implicative order-ca with a distinguished element in the filter, which realizes Pierce's Law.

They give a recipe for constructing, from each Krivine order-ca $\cal A$, an abstract Krivine structure $K_{\cal A}$. And it turns out that the tripos constructed from $K_{\cal A}$ in Streicher's way, is equivalent to the tripos $[-,{\cal A}]$ (theorem 5.15 in \cite{FerrerSantosW:ordcar}). We call such triposes {\em Krivine triposes}.

We follow a different approach, which in our view leads to a simpler representation of Krivine triposes. Let us recall (see \cite{OostenJ:reaics} for details) that in any opca one has a representation of the natural numbers $\{\bar{\cmb{n}}\, |\, n\in\mathbb{N}\}$; since $\bar{\cmb{n}}$ is $\cmb{k}\cmb{s}$-definable, it will be in any filter. Moreover, we have a {\em coding of sequences\/} $[a_0,\ldots ,a_{n-1}]$ (which, again, is $\cmb{k},\cmb{s}$-definable so in the filter whenever $a_0,\ldots ,a_{n-1}$ are). Let us summarize the properties we need in the following lemma:
\begin{lemma}\label{sequencemanage} Let $(A,A')$ be a filtered opca. Then for a standard coding of natural numbers and sequences from $A$, there are elements $\cmb{b},\cmb{c},\cmb{d},\cmb{t}\in A'$ which satisfy:\begin{rlist}
\item For all $n\in\mathbb{N}$ and $k\geq n$, $\cmb{b}\bar{\cmb{n}}[a_0,\ldots ,a_k]\leq a_n$
\item For all $n\in\mathbb{N}$ and $k\geq n$, $\cmb{c}\bar{\cmb{n}}[a_0,\ldots ,a_k]\leq [a_n,\ldots ,a_k]$
\item For all $a\in A$, $\cmb{d}a[a_0,\ldots ,a_{n-1}]\leq [a,a_0,\ldots ,a_{n-1}]$
\item For all $a\in A$, $\cmb{t}a\leq [a]$\end{rlist}\end{lemma}
We can now define an aks out of a filtered opca $(A,A')$ together with a downwards closed subset $U\subseteq A$ which does not meet the filter: $U\cap A'=\emptyset$.
\begin{definition}\label{aksconstrdef}\em Given $(A,A')$ and $U$ as above, we define an aks $K(A,A',U)$ as follows:
\begin{arlist}\item $\Lambda =A$, ${\sf QP}=A'$, $\Pi$ is the set of coded sequences $[a_0,\ldots ,a_{n-1}]$ of $A$.
\item The push operation $\Lambda\times\Pi\to\Pi$ sends $a,\pi$ to $\cmb{d}a\pi$ where $\cmb{d}$ is as in \ref{sequencemanage} iii). We write $a.\pi$ for this.
\item The total binary operation $\Lambda\times\Lambda\to\Lambda$ sends $a,b$ to $\langle\pi\rangle a(b.\pi )$. We write $a{\cdot}b$ for this. Note, that the operation $a,b\mapsto a{\cdot}b$ is total and should not be confounded with the partial operation on $A$ which forms the opca structure; the latter is written $a,b\mapsto ab$, as we have been doing all along.
\item Using the elements $\cmb{b}$ and $\cmb{c}$ from \ref{sequencemanage} i),ii), and writing $\pi _i$ for $\cmb{b}\bar{\cmb{i}}\pi$ and $\pi _{\geq j}$ for $\cmb{c}\bar{\cmb{j}}\pi$, we put:
$$\begin{array}{rcl} \cmb{K} & = & \langle\pi\rangle\pi _0\pi _{\geq 2} \\
\cmb{S} & = & \langle\pi\rangle ((\pi _0{\cdot}\pi _2){\cdot}(\pi _1{\cdot}\pi _2))\pi _{\geq 3} \\
k_{\pi} & = & \langle\rho\rangle\rho _0\pi \\
\cc & = & \langle\pi\rangle\pi _0(k_{\pi _{\geq 1}}.\pi _{\geq 1}) \end{array}$$
\item Finally, the pole $\pole$ is defined by
$$\pole\; =\; \{ (t,\pi )\, |\, t\pi\text{ is defined and }t\pi\in U\}$$
\end{arlist}\end{definition}
\begin{theorem}\label{aksconstrthm} The structure defined in \ref{aksconstrdef} is indeed an abstract Krivine structure.\end{theorem}
\begin{proof} We have to check that the pole satisfies properties (S1)--(S5) from definition~\ref{aksdef}.

For (S1), suppose $(t,s.\pi )\in\pole$, so $t(s.\pi )\in U$. Then $(t{\cdot}s)\pi\in U$ since $(t{\cdot}s)\pi\leq t(s.\pi )$; hence $(t{\cdot}s,\pi )\in\pole$.

For (S2), suppose $(t,\pi )\in\pole$ so $t\pi\in U$. Note that $(t.s.\pi )_0\leq t$ and $(t.s.\pi )_{\geq 2}\leq\pi$, hence
$$\cmb{K}(t.s.\pi )\leq (t.s.\pi )_0((t.s.\pi )_{\geq 2})\leq t\pi$$
so $\cmb{K}(t.s.\pi )\in U$ and therefore $(\cmb{K},t.s.\pi )\in\pole$.

For (S3), suppose $((t{\cdot}u){\cdot}(s{\cdot}u),\pi )\in\pole$, so $((t{\cdot}u){\cdot}(s{\cdot}u))\pi\in U$. Now 
$$\cmb{S}(t.s.u.\pi )\leq ((t{\cdot}u){\cdot}(s{\cdot}u))\pi$$
so $\cmb{S}(t.s.u.\pi )\in U$, hence $(\cmb{S},t.s.u.\pi )\in\pole$.

For (S4), suppose $(t,k_{\pi}.\pi )\in\pole$ so $t(k_{\pi}.\pi )\in U$. Then $\cc (t.\pi )\in U$ since $\cc (t.\pi )\leq t(k_{\pi}.\pi )$. Therefore $(\cc ,t.\pi )\in\pole$.

For (S5), suppose $(t,\pi )\in \pole$ so $t\pi\in U$. We have $k_{\pi}(t.\pi ')\leq t\pi$; hence $k_{\pi}(t.\pi ')\in U$, so $(k_{\pi},t.\pi ')\in\pole$.\end{proof}
\medskip

\noindent Let us denote the aks constructed from $A,A', U$ by ${\cal K}^U_{A,A'}$ and let us call the filtered opca constructed from ${\cal K}^U_{A,A'}$ by Streicher's construction, ${\cal P}(\Pi )^U_{A,A'}$. We wish to compare the tripos $[-,{\cal P}(\Pi )^U_{A,A'}]$ to the tripos $[-,{\cal D}(A,A')]$. First we recall a standard topos-theoretic construction.

For a subset $\alpha$ of $A$ we write $\alpha\to U$ for the set $$\{ a\in A\, |\, \text{for all $b\in\alpha$, $ab$ is defined and $ab\in U$}\}$$ Note that since $U\in {\cal D}A$, $(\alpha\to U)\in {\cal D}A$. For $\phi :I\to {\cal D}A$ we write $\phi\to U$ for the function taking $i\in I$ to $\phi (i)\to U$.
\begin{definition}\label{UBooleanization}\em The {\em Booleanization of the tripos\/} $[-,{\cal D}(A,A')]$ {\em with respect to $U$\/} is the Boolean subtripos of $[-,{\cal D}(A,A')]$ which can be defined in any of the following three equivalent ways:\begin{arlist}
\item For any set $I$, we have the set of functions $\phi :I\to {\cal D}A$ which are isomorphic in $[-,{\cal D}(A,A')]$ to $(\phi\to U)\to U$ (as sub-preorder of $[I,{\cal D}(A,A')]$);
\item For any set $I$, all functions $\phi :I\to {\cal D}A$ but ordered by: $\phi\leq\psi$ if and only if $\phi\leq ((\psi\to U)\to U)$ in $[I,{\cal D}(A,A')]$;
\item For any set $I$, all functions $\phi :I\to {\cal D}A$ but ordered by: $\phi\leq\psi$ if and only if $(\psi\to U)\leq (\phi\to U)$ in $[I,{\cal D}(A,A')]$.\end{arlist}\end{definition}
\begin{theorem}\label{krivinetriposchar} The tripos $[-,{\cal P}(\Pi )^U_{A,A'}]$ is equivalent to the Booleanization of $[-,{\cal D}(A,A')]$ with respect to $U$.\end{theorem}
\begin{proof} Streicher has characterized the preorder in the tripos $[-,{\cal P}_{\pole}(\Pi )]$ arising from an aks, as follows (\cite{StreicherT:kcrcp}, Lemma 5.5): for $\phi ,\psi :I\to {\cal P}_{\pole}(\Pi )$, $\phi\leq\psi$ if and only if there is an element $t\in {\sf QP}$ satisfying:\begin{itemize}\item[] for all $i\in I$, all $u\in |\phi (i)|$ and all $\pi\in\psi (i)$, $(t,u.\pi )\in\pole$\end{itemize}
The first thing to notice is that this preorder extends to $[I,{\cal P}(\Pi )]$ and that in the latter preorder, every $\phi$ is isomorphic to $\phi ^{\pole\pole}$ (both inequalities are realized by $(\cmb{S}{\cdot}\cmb{K}){\cdot}\cmb{K}$); therefore, the tripos $[-,{\cal P}_{\pole}(\Pi )]$ is equivalent to $[-{\cal P}(\Pi )]$ (this was also noticed by Ferrer Santos et al; see 5.15 of \cite{FerrerSantosW:ordcar}).
In our case of ${\cal P}(\Pi )^U_{A,A'}$ we can therefore consider all functions $\phi :I\to {\cal P}(\Pi )$, ordered as follows: $\phi\leq\psi$ if and only if for some $a\in A'$ we have\begin{itemize}\item[$(\circ )$] for all $i\in I$, all $u\in\phi (i)\to U$ and all $\pi\in\psi (i)$, $a(u.\pi )$ is defined and in $U$\end{itemize}
Now if $a\in A'$ satisfies $(\circ )$ then for all $i\in I$, $\langle u\pi\rangle a(u.\pi )$ is an element of $A'$ which is in $(\phi (i)\to U)\to (\psi (i)\to U)$; hence $a\in A'$ realizes $(\phi\to U)\leq (\psi\to U)$ in $[I,{\cal D}(A,A')]$. Conversely, if $a\in A'$ realizes $(\phi\to U)\leq (\psi\to U)$ in $[I,{\cal D}(A,A')]$, then $\langle\rho\rangle a\rho _0\rho _{\geq 1}$ is an element of $A'$ satisfying $(\circ )$.

Furthermore we notice that any element of $[I,{\cal D}(A,A')]$ is isomorphic to a function $\phi :I\to {\cal D}(A,A')$ of the form $i\mapsto\dar X_i$ where $X_i$ is a set of coded sequences: this is easy.

We conclude that any $\phi\in [I,{\cal D}(A,A')]$ of the form $\phi '\to U$ is, up to isomorphism, in the image of the map
$$[I,{\cal P}(\Pi )^U_{A,A'}]\to [I,{\cal D}(A,A')]$$
given by $\phi\mapsto (\phi\to U)$.

Hence, we see that $[-,{\cal P}(\Pi )^U_{A,A'}]$ is equivalent to the {\em opposite\/} of the Booleanization of $[-,{\cal D}(A,A')]$ with respect to $U$. However, since the latter is an indexed pre-Boolean algebra and since every Boolean algebra is isomorphic to its opposite (by the negation map), we have the claimed result.\end{proof}
\begin{theorem}\label{char2} Every Krivine tripos is the Booleanization of $[-,{\cal D}(A,A')]$ with respect to $U$, for some filtered opca $(A,A')$ and a downset $U$ of $A'$ which does not meet $A'$.\end{theorem}
\begin{proof} By Streicher's result, a Krivine tripos is of the form $[-, A]$ for some filtered oca $(A,A')$. By \ref{SsubtriposDS}, it is therefore a subtripos of $[-,{\cal D}(A,A')]$, and in particular a Boolean subtripos. But now by standard topos theory (see Lemma A4.5.21 in \cite{JohnstonePT:skee}), it must be the Booleanization of $[-,{\cal D}(A,A')]$ with respect to some $U$, as required.\end{proof}

\subsection{When is a Krivine tripos localic?}
Recall that Hofstra had characterized, for a BCO $\Sigma$ such that $[-,\Sigma ]$ is a tripos, when this tripos is localic: ${\rm TV}(\Sigma )$ must have a least element (theorem~\ref{localicchar}).

Krivine (\cite{KrivineJL:somprm} formulated a condition for an aks to lead to an interpretation of set theory which is a forcing interpretation: the set
$$|\top\to (\bot\to\bot )|\cap |\bot\to (\top\to\bot )|$$
must contain an element of the set ${\sf QP}$ of quasi-proofs.

Taking into account the way logic is interpreted in an aks, this means the following: for some $a\in {\sf QP}$ we have: \begin{itemize}\item[(Kr)] $\forall s\in\Pi ^{\pole}\forall t,\pi ((a,t.s.\pi )\in\pole\;\text{ and }\; (a,s.t.\pi )\in\pole )$\end{itemize}

\begin{theorem}\label{Kr=Hof} Let $\cal K$ be an aks, an $\Sigma _{\cal K}$ be the filtered oca resulting from $\cal K$ by Streicher's construction. Then ${\cal K}$ satisfies {\rm (Kr)} if and only if ${\rm TV}(\Sigma _{\cal K})$ has a least element.\end{theorem}
\begin{proof} For the only if part, suppose ${\cal K}$ satisfies {\rm (Kr)}. Krivine proved already (see p.\ 16 of \cite{KrivineJL:somprm} that there is a quasi-proof $t$ with the property that for every $X\subseteq\Pi$ and every $b\in {\sf QP}$: if $b\in |X|$, then $t\in |X|$. Since $t\in {\sf QP}$, $\{ t\} ^{\pole}\in\Phi$ (where $\Phi$ is the filter of $\Sigma _{\cal K}$). And for every $\beta\in\Phi$ we have $t\in |\beta |$, so 
$$\beta\subseteq\beta ^{\pole\pole}\subseteq\{ t\} ^{\pole}$$
which, given that the order in $\Sigma$ is {\em reverse\/} inclusion, tells us that ${\rm TV}(\Sigma )$ has a least element.

\noindent Conversely, suppose $\alpha\in\Phi$ is the least element of $\Phi$. Then for all $\beta\in\Phi$, $\beta\subseteq\alpha$, so for every $b\in {\sf QP}$, $\{ b\} ^{\pole}\subseteq\alpha$. If $a\in |\alpha |\cap {\sf QP}$, then $\alpha\subseteq \{ a\} ^{\pole}$, so for all $b\in {\sf QP}$ we have $\{ b\} ^{\pole}\subseteq \{ a\} ^{\pole}$.

Let $\cmb{K}'$ be $\cmb{K}{\cdot}((\cmb{S}{\cdot}\cmb{K}){\cdot}\cmb{K})$; then it is easy to verify that if $(t,\pi )\in\pole$, then $(\cmb{K}',s.t.\pi )\in\pole$, for any term $s$.

Now for $s\in\Pi ^{\pole}$, $\pi\in\Pi$ we have $(s,\pi )\in\pole$ and hence, for any term $t$, we have $(\cmb{K},s.t.\pi )\in\pole$ and $(\cmb{K}',t.s.\pi )\in\pole$, whence $s.t.\pi\in \{\cmb{K}\} ^{\pole}$ and $t.s.\pi\in\{\cmb{K}'\} ^{\pole}$. Since both $\cmb{K}$ and $\cmb{K}'$ are quasi-proofs, by the property of $a$ we find that both $s.t.\pi$ and $t.s.\pi$ are elements of $\{ a\} ^{\pole}$, i.e.\ $(a,s.t.\pi )\in\pole$ and $(a,t.s.\pi )\in\pole$, as desired.\end{proof}
\medskip

\noindent Let us spell out what it means for the tripos $[-,{\cal P}(\Pi )^U_{A,A'}]$ to be localic. The filter consists of those $\alpha\subseteq A$ for which $(\alpha\to U)\cap A'\neq\emptyset$. We require that there is a {\em least\/} such $\alpha$; keeping in mind that the order on ${\cal P}(\Pi )^U_{A,A'}$ is {\em reverse\/} inclusion, we need an $\alpha$ such that\begin{rlist}
\item $(\alpha\to U)\cap A'\neq\emptyset$
\item Whenever $(\beta\to U)\cap\emptyset$, $\beta\subseteq\alpha$\end{rlist}
The following proposition simplifies this somewhat:
\begin{proposition}\label{localiccrit} The tripos $[-,{\cal P}(\Pi )^U_{A,A'}]$ is localic if and only if there exists an element $\cmb{e}\in A'$ with the property that whenever $b\in A'$, $a\in A$ and $ba\in U$, then $\cmb{e}a\in U$.\end{proposition}
\begin{proof} Obvious.\end{proof}

\begin{example}\label{localicexample}\em \begin{arlist}
\item For $U=A-A'$, the tripos $[-,{\cal P}(\Pi )^U_{A,A'}]$ is localic, since $\cmb{e}=\cmb{skk}$ satisfies criterion~\ref{localiccrit}
\item Every filter $A'$ on an opca $A$ induces a preorder $\leq _T$ on $A$ which is analogous to Turing reducibility: $a_1\leq _Ta_2$ if and only if for some $b\in A'$ we have $ba_2\leq a_1$. Note, that $a_1\leq a_2$ implies $a_2\leq _Ta_1$, so for any $a\in A$ the set $\{ b\in A\, |\, a\leq _Tb\}$ is downwards closed w.r.t.\ $\leq$.

Now suppose that the set $U$ is upwards closed w.r.t.\ $\leq _T$ (hence downwards closed w.r.t.\ $\leq$). Then whenever $b\in A'$ and $ba\in U$, we have $ba\leq _Ta$ and $\cmb{skk}a\leq a$ hence $a\leq _T\cmb{skk}a$, so we get $\cmb{skk}a\in U$, which means that again, $\cmb{skk}$ satisfies criterion~\ref{localiccrit} and $[-,{\cal P}(\Pi )^U_{A,A'}]$ is localic.\end{arlist}
\end{example}
We conclude this paper with a family of examples where $[-,{\cal P}(\Pi )^U_{A,A'}]$ is non-localic. We consider the pca ${\cal K}_2$, which is the set of functions $\mathbb{N}\to\mathbb{N}$. Given two such functions $\alpha ,\beta$, we define the relation $\alpha\beta (n)=k$ as: there is a number $N\in\mathbb{N}$ satisfying:
$$\begin{array}{l}\alpha ([n,\beta (0),\ldots ,\beta (N-1)]) = k+1 \\
\text{for all } l<N, \alpha ([n,\beta (0),\ldots ,\beta (l-1)]) = 0\end{array}$$
Here, $[...]$ refers to some computable coding of sequences of natural numbers. We then say: $\alpha\beta$ is defined, if and only if for each $n\in\mathbb{N}$ there is some $k$ such that $\alpha\beta (n)=k$, and $\alpha\beta$ is then the corresponding function $\mathbb{N}\to\mathbb{N}$. This is a partial combinatory algebra.

The pca ${\cal K}_2$ has a filter: the set of total recursive (computable) functions $\mathbb{N}\to\mathbb{N}$. We write $({\cal K}_2,{\rm Rec})$ for the corresponding filtered opca. We are interested in choices for $U$ such that the tripos $[-,{\cal P}(\Pi )^U_{{\cal K}_2,{\rm Rec}}]$ is non-localic. 

We remind the reader of the natural topology on ${\cal K}_2$: basic open sets are of the form
$$V_{\sigma}\; =\;\{\alpha\in {\cal K}_2\, |\, \alpha (0)=\sigma _0,\ldots ,\alpha (n)=\sigma _n\}$$
for some finite sequence $\sigma =(\sigma _0,\ldots ,\sigma _n)$.
\begin{theorem}\label{nonlocalicexmp} Let $U$ be a set of nonrecursive functions. If $U$ is discrete as a subspace of ${\cal K}_2$, then the tripos $[-,{\cal P}(\Pi )^U_{{\cal K}_2,{\rm Rec}}]$ is non-localic.\end{theorem}
\begin{proof} Suppose, for a contradiction, that $\alpha$ is some recursive function with the property that for every recursive $\beta$ and arbitrary $\gamma$, if $\beta\gamma\in U$ then $\alpha\gamma\in U$. First we note that for $\tau\in U$ and $\cmb{skk}\in {\cal K}_2$, which is recursive, $(\cmb{skk})\tau =\tau\in U$, so $\alpha\tau\in U$. Therefore we can fix some $\pi\in {\cal K}_2$ and some $\tau\in U$ such that $\alpha\pi =\tau$.

Since $U$ is discrete, there is some number $N$ such that the basic neighbourhood
$$U_{(\tau (0),\ldots ,\tau (N))}$$
contains no element of $U$ except $\tau$. Let $N'$ be a natural number big enough so that for every $i$, $0\leq i\leq N$, there is some $k<N'$ such that $\alpha ([i,\pi (0),\ldots ,\pi (k-1)])=\tau (i)+1$.
\medskip

\noindent {\sl Claim}. Let $\pi '\in U_{(\pi (0),\ldots ,\pi (N'))}$. Then for every $j\in\mathbb{N}$, either $\alpha\pi '(j)=\tau (j)$ or there is no $k$ such that $\alpha\pi '(j)=k$.

\noindent {\sl Proof of Claim}: suppose $\pi '$ as in the Claim, and $j_0$ such that for some $k\neq\tau (j_0)$ we have $\alpha\pi '(j_0)=k$. Let $t$ be least such that
$$\alpha ([j_0,\pi '(0),\ldots ,\pi '(t-1)])=k+1$$
and let $M={\rm max}(N',t)$. Define $\pi ''\in {\cal K}_2$ as follows:
$$\pi ''(i)\; =\;\left\{\begin{array}{rl} \pi '(i) & \text{if } i\leq M \\ \tau (i-(M+1)) & \text{otherwise}\end{array}\right.$$
Clearly, there is some recursive function $\beta$ such that $\beta\pi ''=\tau\in U$; hence, $\alpha\pi ''\in U$, but by construction we must have $\alpha\pi ''=\tau$, but this contradicts the fact that $\alpha\pi ''(j_0)\neq \tau (j_0)$. This proves the claim.
\medskip

\noindent But now, with $\alpha$ recursive and the finite sequence $(\pi (0),\ldots ,\pi (N'))$ given, we have a recipy to compute $\tau$: for any input $j$, either there is some $k\leq N'$ such that $\alpha ([j,\pi (0),\ldots ,\pi (k)]>0$ (and then for the least such $k$, this must be $\tau (j)+1$); or there is some sequence $(n_0,\ldots ,n_m)$ which is minimal such that
$$\alpha ([j,\pi (0),\ldots ,\pi (N'),n_0,\ldots ,n_m])>0$$
and then, by the claim, the result must be $\tau (j)+1$. This algorithm contradicts the assumption that $\tau\in U$, and hence non-computable.\end{proof}

\begin{small}
\bibliographystyle{plain}

\end{small}
\end{document}